\documentclass[10pt]{article}
\usepackage[english]{babel}
\usepackage{amsmath,amsthm}
\usepackage{amsfonts}
\usepackage{enumerate}

\usepackage{graphicx}
\usepackage{color}

\usepackage[top=1.40in, bottom=1.39in, left=1.67in, right=1.67in]{geometry}

\newtheorem{thm}{Theorem}[section]
\newtheorem{cor}[thm]{Corollary}
\newtheorem{lem}[thm]{Lemma}
\newtheorem{prop}[thm]{Proposition}
\theoremstyle{definition}
\newtheorem{defn}[thm]{Definition}

\newtheorem{example}[thm]{Example}
\theoremstyle{remark}

\newtheorem{rem}[thm]{Remark}

\numberwithin{equation}{section}

\newcommand{\norm}[1]{\lVert #1 \rVert^2}

\newcommand{\spin}{\ifmmode{\rm Spin}\else{${\rm spin}$\ }\fi}
\newcommand{\spinc}{\ifmmode{{\rm Spin}^c}\else{${\rm spin}^c$\ }\fi}
\newcommand{\spinct}{\mathfrak t}
\newcommand{\spincs}{\mathfrak s}

\begin{document}

\title{Non-integer surgery and branched double covers of alternating knots}%
\author{Duncan McCoy}%
\author{Duncan McCoy\\
{\small University of Glasgow}\\
{\small\texttt{d.mccoy.1@research.gla.ac.uk}}}%
\date{}%
\maketitle
\begin{abstract}
We show that if the branched double cover of an alternating link arises as $p/q \in \mathbb{Q}\setminus \mathbb{Z}$ surgery on a knot in $S^3$, then this is exhibited by a rational tangle replacement in an alternating diagram.
\end{abstract}
\section{Introduction}
Given a knot or link in $L \subset S^3$, one can obtain a new link by {\em rational tangle replacement}, that is by replacing one rational tangle in $L$ with some other rational tangle. If $L'$ is obtained from $L$ by rational tangle replacement then the branched double cover $\Sigma(L')$ can be obtained by surgery on some knot in $\Sigma(L)$ \cite{montesinos1973variedades}. This correspondence between tangle replacement and surgery is frequently referred to as the {\em Montesinos trick}. One common occurrence of this correspondence comes from crossing changes, which arise in the study of unknotting and unlinking numbers. In this case, the resulting surgery is of half-integer slope. For alternating knots with unknotting number one, there is a converse to the Montesinos trick.
\begin{thm}[\cite{mccoy2013alternating}, Theorem 1]\label{thm:unknotting}
For an alternating knot, $K$, the following are equivalent:
\begin{enumerate}[(i)]
\item $u(K)=1$;
\item The branched double cover, $\Sigma(K)$, can be obtained by half-integer surgery on a knot in $S^3$;
\item $K$ has an unknotting crossing in any alternating diagram.
\end{enumerate}
\end{thm}
One key ingredient in the proof of Theorem~\ref{thm:unknotting}, was the obstruction to half-integer surgery developed by Greene \cite{Greene3Braid}, which combines Donaldson's Theorem A and correction terms from Heegard Floer homology. This has been generalized by Gibbons to give an obstruction to surgeries of other slopes \cite{gibbons2013deficiency}. This allows us to generalize Theorem~\ref{thm:unknotting} to non-integer surgeries.

\begin{thm}\label{thm:informaltheorem}
Let $L$ be an alternating knot or link. For any $p/q\in \mathbb{Q}\setminus\mathbb{Z}$ with $|p/q|>1$, the double branched cover $\Sigma(L)$ arises as $p/q$-surgery on a knot in $S^3$ if and only if $L$ possesses an alternating diagram obtained by rational tangle replacement from an almost-alternating diagram of the unknot and the surgery corresponding to this replacement is of slope $p/q$.
\end{thm}
We prove the following theorem, which gives a more precise description of the tangle replacements in Theorem~\ref{thm:informaltheorem}.
\begin{thm} \label{thm:rationalsurgery}
Let $L$ be an alternating link. Let $p>q>1$ be coprime integers and suppose that $p/q=n-r/q$, where $0<r<q$. The following are equivalent:
\begin{enumerate}[(i)]
\item there is $\kappa \subset S^3$ with $\Sigma(L)=S^3_{-p/q}(\kappa)$;
\item for any alternating diagram $D$ of $L$ the Goeritz form $\Lambda_D$ of $D$ is isomorphic to a $p/q$-changemaker lattice;
\item $L$ has an alternating diagram, $\widetilde{D}$, containing a rational tangle $T$ of slope $\frac{q-r}{r}$. Furthermore, if $D'$ is the alternating diagram obtained by replacing $T$ with a single crossing $c$, then $c$ is an unknotting crossing which is positive if $\sigma(D')=-2$ and negative if $\sigma(D')=0$.
\end{enumerate}
\end{thm}
The condition that $p>q$ in Theorem~\ref{thm:rationalsurgery} is necessary for $(iii)$ to hold. However, when $|p/q|\leq 1$ the situation is even simpler. The following is a straightforward consequence of known bounds on $L$-space surgeries.
\begin{prop}\label{prop:smallsurgery}
Suppose that $S^3_{-p/q}(\kappa)=\Sigma(L)$ for some $1>p/q>0$ with $L$ an alternating link, then $\kappa$ is the unknot. In particular $\Sigma(L)$ is a lens space and $L$ is a 2-bridge knot or link.
\end{prop}

\paragraph{}The starting point for the proof of Theorem~\ref{thm:rationalsurgery}, is a result of Gibbons \cite{gibbons2013deficiency}, which generalizes the work of Greene \cite{greene2010space}, \cite{GreeneLRP}, \cite{Greene3Braid}. It provides strong restrictions on when a manifold $Y$ can be constructed by negative surgery and be the boundary of a sharp positive-definite 4-manifold. In its original formulation, this result has some technical hypotheses on the $d$-invariants of $Y$. However, the work of Ni and Wu \cite{ni2010cosmetic} and Greene \cite{greene2010space} is adequate to show that these hypotheses are automatically satisfied. This gives the following refinement of Gibbons' theorem.

\begin{thm}\label{thm:Gibbonsrefined}
Let $\kappa \subset S^3$ be a knot and suppose that for some $p/q>0$, $S^3_{-p/q}(\kappa)$ bounds a positive-definite, simply-connected, smooth 4-manifold $X$ with intersection form $Q_X$. If the manifold $X$ is sharp, then $Q_X$ is isomorphic to a $p/q$-changemaker lattice.
\end{thm}

The definition of a $p/q$-changemaker lattice is given in Section~\ref{sec:cmlattices}.

\subsection{Rational unknotting number one}
A knot or link $L$ has {\em rational unknotting number one} if it has a diagram which can be turned into an unknot by replacing a rational tangle with a trivial tangle \cite{Line96knots}. Equivalently, $L$ has rational unknotting number one if it has some diagram which can be turned into the unknot by some rational tangle replacement. Using this terminology, Theorem~\ref{thm:rationalsurgery} yields the following corollaries.
\begin{cor}\label{cor:TUNO1}
If $L$ is an alternating link and there is $\kappa \subset S^3$  such that $\Sigma(L)=S_{p/q}^3(\kappa)$ for some $q\geq 2$, then $K$ has rational unknotting number one.
\end{cor}
\begin{proof}
If $|p/q|\leq 1$, then Proposition~\ref{prop:smallsurgery} shows that $L$ is a 2-bridge knot or link. Any 2-bridge knot or link obviously has rational unknotting number one. Assume now that $|p/q|>1$.
If $\Sigma(L)=S_{p/q}^3(\kappa)$, then $\Sigma(\overline{L})=S_{-p/q}^3(\overline{\kappa})$. Thus Theorem~\ref{thm:rationalsurgery} applies to show $L$ or $\overline{L}$ has rational unknotting number one. The result follows as $L$ has rational unknotting number one if and only if $\overline{L}$ has.
\end{proof}

\begin{cor}\label{cor:TUNO2}
If $L$ is an alternating link with rational unknotting number one, then either $\Sigma(L)$ can be constructed by integer surgery on a knot in $S^3$ or it has an alternating diagram which can be obtained by rational tangle replacement on an almost-alternating diagram of the unknot.
\end{cor}
\begin{proof}
Montesinos' trick implies that if $L$ has rational unknotting number one then $\Sigma(L) =S_{p/q}^3(\kappa)$ for some $p/q$ and $\kappa \subset S^3$. If $q=1$, then $\Sigma(L)$ can be constructed by integer surgery. Assume now that $q\geq 2$. If $|p/q|>1$, then Theorem~\ref{thm:rationalsurgery} applies to $L$ or $\overline{L}$ to give the necessary alternating diagram $D$. If $|p/q|<1$, then $L$ is a 2-bridge link and any alternating 2-bridge diagram has the required property. In fact, any 2-bridge diagram can be obtained by rational tangle replacement on the standard diagram of the unknot.
\end{proof}

\subsection{Outline of Theorem~\ref{thm:rationalsurgery}'s proof}
We will now outline the proof of Theorem~\ref{thm:rationalsurgery}. The equivalence between tangle replacement and surgery on the double branched cover in $(iii) \Rightarrow (i)$ comes from the Montesinos trick. We sketch the details of this in Section~\ref{sec:tangles}. The implication $(i) \Rightarrow (ii)$ follows from combining the work of Ozsv{\'a}th and Szab{\'o} on the branched double covers of alternating knots \cite{ozsvath2005heegaard} with the work of Gibbons \cite{gibbons2013deficiency}. In Section~\ref{sec:deficiencies}, we make the necessary observations on $d$-invariants to show that Gibbons' work is sufficient to prove Theorem~\ref{thm:Gibbonsrefined}. To establish $(ii) \Rightarrow (iii)$, we study  knots and links with diagrams for which the Goeritz form is isomorphic to a $p/q$-changemaker lattice.

\paragraph{} If $p/q$ can be written in the form $p/q=n-r/q>1$ with $q>1$, then it has a continued fraction $[a_0, a_1, \dotsc , a_l]^-$ with $a_i\geq 2$ for $1\leq i\leq l$ and $a_0=n$. A $p/q$-changemaker lattice takes the form of an  orthogonal complement (see Section~\ref{sec:cmlattices} for the full definition):
$$L=\langle w_0, \dotsc , w_l \rangle^\bot \subseteq \mathbb{Z}^{t+s+1}=\langle f_1, \dotsc , f_t, e_0, \dotsc , e_s \rangle,$$
where the $f_i$ and $e_j$ form an orthonormal basis for $\mathbb{Z}^{r+s+1}$, and the $w_i$ have the properties that
\[w_i \cdot w_j =
    \begin{cases}
        a_i   & i=j\\
        -1   & |i-j|=1 \\
        0   &  |i-j|\geq 2,
    \end{cases}
\]
and
$$w_0\cdot e_i = 0 \quad \text{for } 1\leq i \leq s,$$
$$w_j\cdot f_i = 0 \quad \text{for } 1\leq i \leq t\text{, and } 1\leq j \leq l.$$
In addition, $w_0$ is required to satisfy further conditions called the changemaker conditions.

Given an alternating diagram, $D$, with a chessboard colouring, let $\Lambda_D$ be the lattice defined by the associated Goeritz matrix. An isomorphism between $\Lambda_D$ and a changemaker lattice, $L$, allows us to label each unshaded region of $D$ by a vector in $L$. After performing some flypes, we show that the regions labeled by vectors satisfying $v\cdot e_i \ne 0$ for some $0\leq i \leq s$, determine a rational tangle of the required slope. If we replace this rational tangle with a single crossing $c$ to obtain an alternating diagram $D'$, then we will show that the Goeritz form of $D'$ is isomorphic to an $(n-1/2)$-changemaker lattice. This allows us to appeal to \cite{mccoy2013alternating} to show that $c$ is an unknotting crossing for $D'$ of the required sign.

\subsection{Further remarks}
The most obvious limitation of Theorem~\ref{thm:informaltheorem} is that it only describes non-integer surgeries. It is natural to wonder whether there is a similar characterisation of integer surgeries yielding double branched covers of alternating knots. The techniques in this paper do not extend to this case in any easy way. However, the author has not yet been unable to exhibit an example which shows that the integer case is different. So, although the possibility remains that an analogue of Theorem~\ref{thm:informaltheorem} may hold for integer surgeries, there are substantial algebraic complications to be dealt with in the proof of such a result.

\paragraph{}It also seems natural to ask about the class of knots which surger to give the double branched covers of alternating links. There are some cases for which this has been answered. For example, the cyclic surgery theorem of Culler, Gordon, Luecke and Shalen shows that torus knots are the only ones for which non-integer surgery can yield a lens space, i.e. the double-branched cover of a 2-bridge knot or link \cite{cglscyclic}. In general, this seems to be a hard question although it seems possible to give restrictions on the surgery slopes of a given knot which can give the double branched cover of an alternating link. We plan to return to this in future work.

\subsection{Acknowledgements}
The author would like to thank his supervisor, Brendan Owens, for his general guidance and helpful comments on this paper.

\section{Surgeries bounding sharp 4-manifolds}\label{sec:deficiencies}
In \cite{gibbons2013deficiency}, Gibbons proves a generalisation of the changemaker theorems appearing in \cite{Greene3Braid}, \cite{greene2010space} and \cite{GreeneLRP}. Before we state the theorem, we will briefly recall some of the necessary background on the $d$-invariants of Heegaard Floer homology.

\paragraph{}For $Y$ a rational homology 3-sphere, its Heegaard Floer homology splits as a direct sum over its \spinc-structures:
$$\widehat{HF}(Y)\cong \bigoplus_{\spincs \in \spinc(Y)}\widehat{HF}(Y,\spincs).$$
Associated to each summand there is a numerical invariant $d(Y,\spincs)\in \mathbb{Q}$, called the {\em $d$-invariant} \cite{Ozsvath03Absolutely}. If $Y$ is the boundary of a negative-definite 4-manifold $X$, then for any $\spinct \in \spinc(X)$ which restricts to $\spincs \in \spinc(Y)$ there is a bound on the  $d$-invariant:
\begin{equation}\label{eq:sharpdef}
c_1(\spinct)+b_2(X)\leq 4d(Y,\spincs).
\end{equation}
We say that $X$ is {\em sharp} if for every $\spincs \in \spinc(Y)$ there is some $\spinct \in \spinc(X)$ which restricts to $\spincs$ and attains equality in \eqref{eq:sharpdef}.

\paragraph{}Let $\kappa \subset S^3$ be a knot. For fixed $p/q \in \mathbb{Q}\setminus \{0\}$, there are canonical identifications \cite{ozsvath2011rationalsurgery}:
$$\spinc(S_{p/q}^3(\kappa)) \leftrightarrow \mathbb{Z}/p\mathbb{Z} \leftrightarrow \spinc(S_{p/q}^3(U)).$$
Using these identifications we are able define $\tilde{d}_i \in \mathbb{Q}$ for each $i \in \mathbb{Z}/p\mathbb{Z}$ by the formula
$$\tilde{d}_i:=d(S_{p/q}^3(\kappa),i)-d(S_{p/q}^3(U),i).$$
Now consider the values for which $\tilde{d}_i$ vanishes:
$$Z=\{0\leq i \leq p-1 \colon \tilde{d}_i =0\}.$$

\paragraph{}The generalisation of the changemaker theorems we are interested in is \cite[Theorem 1.2]{gibbons2013deficiency}.
\begin{thm}[Gibbons]\label{thm:Gibbons}
For $p/q>0$, suppose that $S_{p/q}^3(\kappa)$ bounds a negative-definite simply-connected 4-manifold, $X$, with intersection lattice $Q_X$. If $X$ is sharp and $|Z|>\min\{p-1,q\}$, then $Q_X$ is isomorphic to a $p/q$-changemaker lattice.
\end{thm}

Gibbons originally established Theorem~\ref{thm:Gibbons} for smaller values of $|Z|$ than we have stated here. However, since the condition $S_{p/q}(\kappa)$ bounds a negative-definite manifold is sufficient to imply our stronger requirement on $|Z|$, Theorem~\ref{thm:Gibbons} is sufficient for our purposes.

\paragraph{}The work of Ni and Wu shows that for $0\leq i \leq p-1$, the values $\tilde{d}_i$ may be calculated by the formula \cite[Proposition 1.6]{ni2010cosmetic},
\begin{equation}\label{eqn:NiWuHV}
\tilde{d}_i=-2\max\{V_{\lfloor \frac{i}{q} \rfloor},H_{\lfloor \frac{i-p}{q} \rfloor}\},
\end{equation}
where $V_j$ and $H_j$ are sequences of positive integers depending only on $\kappa$, which are non-increasing and non-decreasing respectively. These further satisfy $H_{-j}=V_j=0$ for $j\geq g(\kappa)$, where $g(\kappa)$ is the genus of $\kappa$. In fact, it can be shown that $V_j=H_{-j}$ for all $j$ \cite[Proof of Theorem~3]{owensstrle2013immersed}. Using these properties of the $V_j$ and $H_j$, \eqref{eqn:NiWuHV} can be rewritten as
\begin{equation}\label{eqn:NiWuVonly}
\tilde{d}_i=-2V_{\min \{ \lfloor \frac{i}{q}\rfloor , \lceil \frac{p-i}{q}\rceil \}}.
\end{equation}
\paragraph{}Let $\tilde{g}(\kappa)\geq 0$ be the minimal integer such that $V_{\tilde{g}}=0$. Such an integer necessarily exists and is at most $g(\kappa)$. In fact, if $\kappa$ is an $L$-space knot, then $\tilde{g}(\kappa)= g(\kappa)$.

We can compute the size of $Z$ in terms of  $\tilde{g}(\kappa)$.
\begin{lem}\label{lem:vanishingcountlemma}
The size of $Z$ is given by
$$|Z|=
\begin{cases}
  p          & \text{if } \tilde{g}= 0 \\
  p-(2\tilde{g}-1)q  & \text{if } p/q>2\tilde{g}-1>0\\
  0  & \text{if } p/q\leq 2\tilde{g}-1.
\end{cases}$$
\end{lem}
\begin{proof}
For $0\leq i \leq p-1$, \eqref{eqn:NiWuVonly} shows that $\tilde{d}_i=0$, if and only if $\min \{\lfloor \frac{i}{q}\rfloor , \lceil \frac{p-i}{q}\rceil \} \geq \tilde{g}$. If $\tilde{g}(\kappa)=0$, then this is true for all $i$. If $\tilde{g}(\kappa)>0$, then $Z$ consists of the values in the range,
$$\tilde{g}q\leq i \leq p+q-\tilde{g}q-1.$$
This range is nonempty only if $p/q> 2\tilde{g}-1$, in which case it contains $p-(2\tilde{g}-1)q$ values.
\end{proof}

If $\kappa$ is an $L$-space knot and $S_{n}^3(\kappa)$ bounds a simply-connected, negative-definite smooth manifold for some integer $n>0$, then Greene has shown the bound \cite[Theorem 1.1]{greene2010space},
\begin{equation}\label{eq:greenebound}
2g(\kappa)\leq n-\sqrt{n}.
\end{equation}
A similar result holds when the $L$-space condition is omitted.
\begin{lem}\label{lem:Greenerefined} Suppose that $S_{p/q}^3(\kappa)$, bounds a smooth, simply-connected, negative-definite manifold, $X$, then
$$2\tilde{g}(\kappa) \leq n-\sqrt{n},$$
where $n=\lceil p/q \rceil > 0$.
\end{lem}
\begin{proof}[Proof (sketch)]
There is a positive-definite cobordism $W$ from $S_{n}^3(\kappa)$ to $S_{p/q}^3(\kappa)$, obtained by attaching 2-handles to $S_{n}^3(\kappa)\times I$. Gluing $\overline{W}$ to $X$ along $S_{p/q}^3(\kappa)$ shows that $S_{n}^3(\kappa)$ also bounds a simply-connected, negative-definite smooth manifold.

\paragraph{}Now we consider the proof of \eqref{eq:greenebound} \cite[Theorem 1.1]{greene2010space}. We observe that the $L$-space condition is only required to show that
$$d(S_{n}^3(\kappa),i)-d(S_{n}^3(U),i)\leq 0,$$
for all $i$ with equality if and only if $i\geq g(\kappa)$. However, \eqref{eqn:NiWuVonly} shows that in general,
$$d(S_{n}^3(\kappa),i)-d(S_{n}^3(U),i)\leq 0,$$
for all $i$, with equality if and only if $i\geq \tilde{g}(\kappa)$. So the same argument shows
$$2\tilde{g}(\kappa) \leq n-\sqrt{n},$$
as required.
\end{proof}
Now we can combine Lemma~\ref{lem:Greenerefined} and Lemma~\ref{lem:vanishingcountlemma} to show that Theorem~\ref{thm:Gibbonsrefined} follows from Theorem~\ref{thm:Gibbons}.

\begin{proof}[Proof of Theorem~\ref{thm:Gibbonsrefined}]
We need only show that the hypothesis on the set $Z$ is satisfied.
Let $n=\lceil p/q \rceil$. If $n=1$ or 2, then Lemma~\ref{lem:Greenerefined} shows that
$\tilde{g}=0$, which implies that $|Z|=p$. If $n>2$, then Lemma~\ref{lem:Greenerefined} and Lemma~\ref{lem:vanishingcountlemma} give
 $$|Z|\geq p+q-qn+\sqrt{n}q>\sqrt{n}q>q.$$
 This shows $|Z|>\min \{p-1,q\}$, as required. Theorem~\ref{thm:Gibbonsrefined} follows, since  $$-{S_{p/q}^3(\kappa)}=S_{-p/q}^3(\overline\kappa).$$
\end{proof}

\section{Graph lattices}
We briefly recall the definition of a graph lattice and a few of their key properties. All the results in this section can be found complete with proof in \cite{mccoy2013alternating}.

\paragraph{} Let $G=(V,E)$ be a finite, connected, undirected graph with no self-loops. For a pair of disjoint subsets $R,S \subset V$, let $E(R,S)$ be the set of edges between $R$ and $S$. Define $e(R,S)=|E(R,S)|$. We will use the notation $d(R)=e(R,V\setminus R)$.

\paragraph{} Let $\overline{\Lambda}(G)$ be the free abelian group generated by $v\in V$. Define a symmetric bilinear form on $\overline{\Lambda}(G)$ by
\[
v\cdot w =
  \begin{cases}
   d(v)            & \text{if } v=w \\
   -e(v,w)       & \text{if } v\ne w.
  \end{cases}
\]
In this section we will use the notation $[R]=\sum_{v\in R}v$, for $R\subseteq V$. The above definition gives
\[
v\cdot [R] =
  \begin{cases}
   -e(v,R)            & \text{if } v\notin R \\
   e(v,V\setminus R)       & \text{if } v\in R.
  \end{cases}
\]

From this it follows that $[V]\cdot x= 0$ for all $x \in \overline{\Lambda}(G)$. We define the {\em graph lattice} of $G$ to be
$$\Lambda(G):= \frac{\overline{\Lambda}(G)}{\mathbb{Z}[V]}.$$
The bilinear form on $\overline{\Lambda}(G)$ descends to $\Lambda(G)$. Since we have assumed that $G$ is connected, the pairing on $\Lambda(G)$ is positive-definite. This makes $\Lambda(G)$ into an integral lattice. Henceforth, we will abuse notation by using $v$ to denote its image in $\Lambda(G)$.

We have the following useful bound:
\begin{lem}\label{lem:usefulbound}
Let $x=[R]$ be a sum of vertices, then for any $z\in \Lambda(G)$, we have
$$(x-z)\cdot z\leq 0.$$
\end{lem}
We say that $z \in \Lambda(G)$ is {\em irreducible} if we cannot find $x,y \in \Lambda(G) \setminus \{0\}$ such that $z=x+y$ and $x\cdot y \geq 0$. The irreducible vectors of $\Gamma(G)$ can be characterised as  follows.
\begin{lem}\label{lem:irreducible}
The vector $x \in \Lambda(G)\setminus \{0\}$ is irreducible if, and only if, $x=[R]$ for some $R\subseteq V$ such that $R$ and $V\setminus R$ induce connected subgraphs of $G$.
\end{lem}

Recall that a connected graph is {\em 2-connected} if it can not be disconnected by deleting a vertex. This property is equivalent to $\Lambda(G)$ being {\em indecomposable}, that is, $\Lambda(G)$ cannot be written as the orthogonal direct sum $\Lambda(G)= L_1 \oplus L_2$ with $L_1,L_2$ non-zero sublattices.
\begin{lem}\label{lem:2connectgraphlat}
The following are equivalent:
\begin{enumerate}[(i)]
\item The graph $G$ is 2-connected;
\item Every vertex $v\in V$ is irreducible;
\item The lattice $\Lambda(G)$ is indecomposable.
\end{enumerate}
\end{lem}

The following will also be useful.
\begin{lem}\label{lem:cutedge}
Suppose that $G$ is 2-connected, contains no cut-edges and there is a vertex $v$ such that we can find $x,y\in \Lambda(G)$, with $v=x+y$ and $x\cdot y=-1$. Then there is a cut edge $e$ in $G\setminus \{v\}$ and if $R,S$ are the vertices of the two components of $(G\setminus \{v\})\setminus\{e\}$ then $\{x,y\}=\{[R]+v,[S]+v\}$. Furthermore, there are unique vertices $u_1,u_2\ne v$, with $x\cdot u_1=y\cdot u_2=1$, and any vertex $w \notin \{v,u_1,u_2\}$ satisfies $w\cdot x,w\cdot y\leq 0$.
\end{lem}

\section{Alternating diagrams}\label{sec:altdiagrams}
Given a diagram $D$ of a link $L$, we get a division of the plane into connected regions. We may colour these regions black and white in a chessboard manner. There are two possible choices of colouring, and each gives an incidence number, $\mu(c)\in \{\pm 1\}$, at each crossing $c$ of $D$, as shown in Figure~\ref{fig:incidencenumber}.
 \begin{figure}[h]
  \centering
  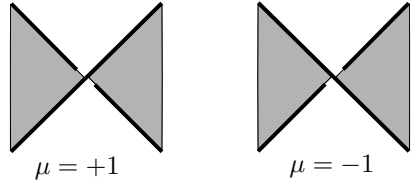
 \caption{The incidence number of a crossing.}
 \label{fig:incidencenumber}
\end{figure}
 We construct a planar graph, $\Gamma_D$, by drawing a vertex in each white region and an edge $e$ for every crossing $c$ between the two white regions it joins and we define $\mu(e):=\mu(c)$. We call this the {\em white graph} corresponding to $D$. This gives rise to a {\em Goeritz matrix}, $G_D=(G_{ij})$, defined by labeling the vertices of $\Gamma_D$, by $v_1,\dotsc , v_{r+1}$ and for $1\leq i,j \leq r$, set
 $$g_{ij}=\sum_{e \in E(v_i,v_j)}\mu(e)$$
 for $i\ne j$ and
 $$g_{ii} = - \sum_{e \in E(v_i, \Gamma_D\setminus v_i)}\mu(e)$$
 otherwise \cite[Chapter 9]{lickorish1997introduction}.

 \paragraph{} Now suppose that $L$ is an alternating, non-split link. If $D$ is any alternating diagram, then we may fix the colouring so that $\mu(c)=-1$ for all crossings. In this case, $G_D$ defines a positive-definite bilinear form. This in turn gives a lattice, $\Lambda_D$ which we will refer to as the {\em white lattice} of $D$. Observe that if $D$ is reduced (i.e. contains no nugatory crossings), then $\Gamma_D$ contains no self-loops or cut-edges and $\Lambda_D$ is isomorphic to the graph lattice $\Lambda(\Gamma_D)$. Moreover, if $D$ and $D'$ are any two reduced diagrams, then one can be obtained from another by a sequence of flypes. The following lemma, which is proven in \cite{mccoy2013alternating}, allows us to detect certain flypes algebraically and provides an explicit isomorphism by relating the vertices of $\Gamma_D$ and $\Gamma_{D'}$. The flypes in question are depicted in Figure~\ref{fig:flype1}.
 \begin{lem}\label{lem:flype1}
 Let $D$ be a reduced alternating diagram. Suppose $\Gamma_D$ is 2-connected and has a vertex $v$, which can be written as $v=x+y$, for some $x,y \in \Lambda_D$ with $x\cdot y=-1$. Then there are unique vertices, $u_1,u_2\ne v$, satisfying $u_1\cdot x, u_2\cdot y>0$, and there is a flype to a diagram $D'$ and an isomorphism $\Lambda_{D'}\cong \Lambda_{D}$, such that the vertices of $\Gamma_{D'}$ are obtained from those of $\Gamma_D$ by replacing $v,u_1$ and $u_2$, with $x,y$ and $u_1+u_2$.
 \end{lem}
  \begin{figure}[h]
  \centering
  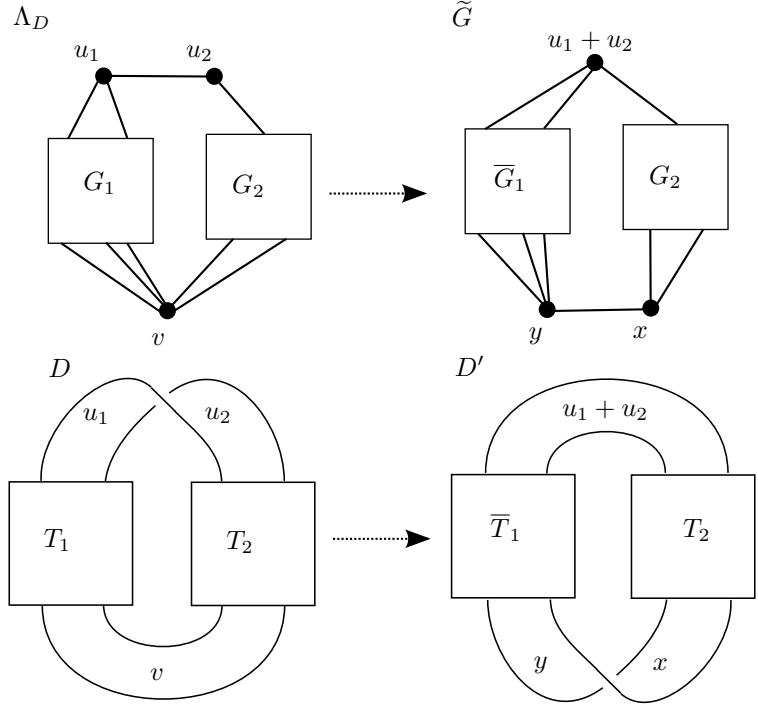
 \caption{The flype and new embeddings given by Lemma \ref{lem:flype1}, where $\overline{G}_1$, is the reflection of $G_1$ about a vertical line and $\overline{T}_1$ is the corresponding tangle (which can also be obtained by rotating $T_1$ by an angle of $\pi$ about a vertical line).}
 \label{fig:flype1}
\end{figure}
The other type of flype we will wish to consider is of the type shown in Figure~\ref{fig:flype2}.
Suppose that we have an alternating diagram $D$ with white regions $v$ and $w$ which form a cut set in $\Gamma_D$ and have a crossing between them. If $G_1$ is a component of $\Gamma_D \setminus \{v,w\}$, which is adjacent to an edge between $v$ and $w$ in the plane, then Figure~\ref{fig:flype2} shows that there is a flype in $D$ which rotates the tangle corresponding to $G_1$ by $\pi$. If $D'$ is the diagram resulting from this flype, then the lattice $\Lambda_{D'}$ is isomorphic to $\Lambda_D$. Since
$$(v+w)\cdot z = - z.[G_1],$$
for all $z \in G_1$, an isomorphism from $\Lambda_D$ to $\Lambda_{D'}$ can be given by replacing each vertex $z \in G_1$ by $-z$ and by replacing $v$ and $w$ by $v+[G_1]$ and $w+[G_1]$ respectively.

\begin{figure}[h]
  \centering
  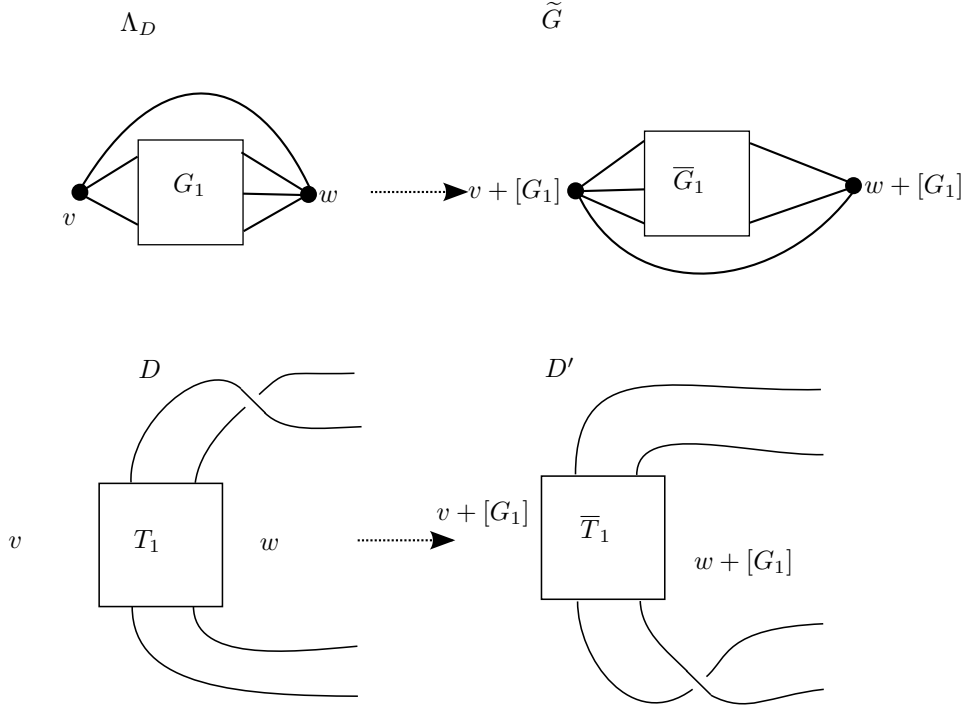
 \caption{Flypes arising when $v$ and $w$ form a cutset with an edge between them. Here $\overline{G}_1$ is the reflection of $G_1$ about a vertical line and $\overline{T}_1$ is the corresponding tangle (which can also be obtained by rotating $T_1$ by an angle of $\pi$ about a vertical line).}
 \label{fig:flype2}
\end{figure}

\section{Rational tangles}\label{sec:tangles}
We will give a summary of the necessary background material on rational tangles and their slopes. A more comprehensive account can be found in \cite{Gordon09dehnsurgery} or \cite{BurdeZieschang}.
A {\em tangle} in $B^3$ is the pair $(B^3,A)$ where $A$ is a properly embedded 1-manifold. We say $(B^3,A)$ is {\em marked}, if $\partial B^3 \cap A$ consists of 4 points and we have fixed an identification of the pairs $(\partial B^3, \partial B^3 \cap A)$ and $(S^2, \{NE,NW,SE,SW\})$. This marking on the boundary is illustrated in Figure~\ref{fig:handv}. Two marked tangles $(B^3,A)$ and $(B^3,A')$ are considered equal if there is an isotopy of $B$, fixing the boundary $\partial B$, which takes $A$ to $A'$.
\paragraph{}Consider the marked tangles $R(0/1)$ and $R(0/1)$ and the tangle operations $h$ and $v$, as shown in Figure~\ref{fig:handv}.

 \begin{figure}[h]
  \centering
  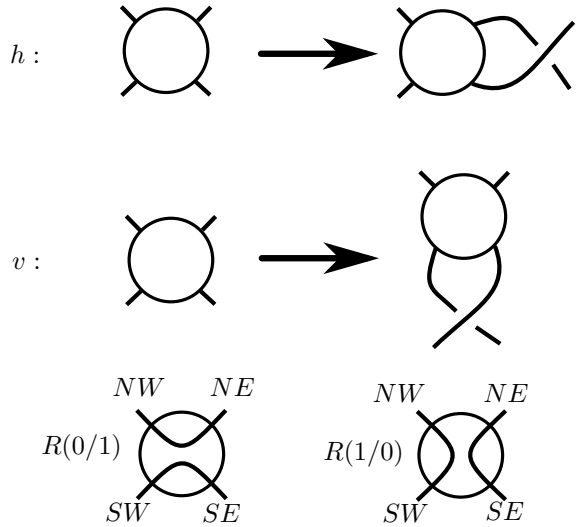
 \caption{The rational tangles $R(0/1)$, $R(1/0)$ and the tangle-building operations $h$ and $v$.}
 \label{fig:handv}
\end{figure}

For non-zero integers $a_1, \dotsb, a_k$, we get a rational number $p/q$ via the continued fraction
$$p/q = [a_1 , \dotsc, a_k]^+=
       a_1 + \cfrac{1}{a_2
           + \cfrac{1}{\ddots
           + \cfrac{1}{a_k} } }.
$$
This allows us to construct the {\em rational tangle} of slope $p/q$ defined by
\begin{equation}\label{eq:standardtangle}
R(p/q) =
    \begin{cases}
        h^{a_1}v^{a_2} \dotsb h^{a_{k-1}}v^{a_k}R(1/0)   & \text{if $k$ even}\\
        h^{a_1}v^{a_2} \dotsb v^{a_{k-1}}h^{a_k}R(0/1)   & \text{if $k$ odd.}
    \end{cases}
\end{equation}
As unmarked tangles, every rational tangle is homeomorphic to the trivial tangle, $R(1/0)$. Conversely, every tangle homeomorphic to $R(1/0)$ is equivalent as a marked tangle to a rational tangle. It is a theorem of Conway that the rational tangles are determined as marked tangles by the value $p/q$ \cite{conway69algebraic}.

\begin{figure}[h]
  \centering
  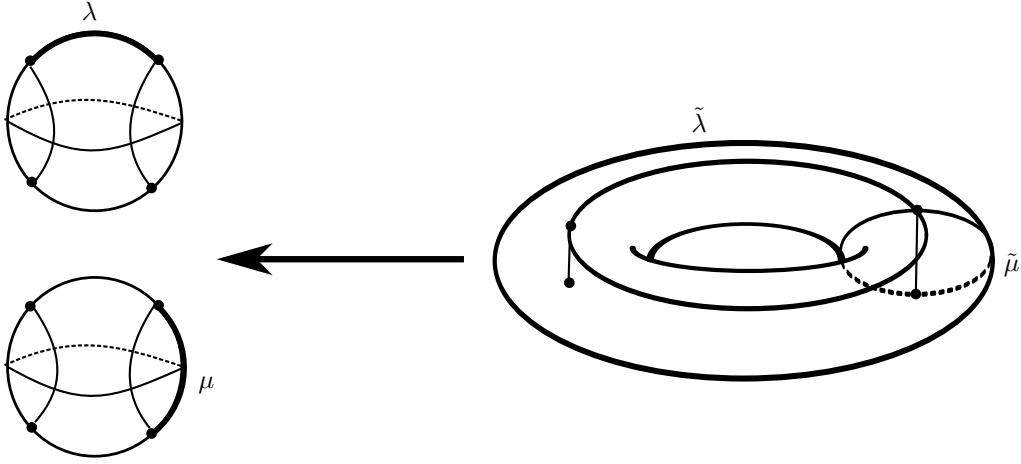
 \caption{Branched double cover over $R(1/0)$, the curves $\tilde{\mu}$ and $\tilde{\lambda}$ are the lifts of $\mu$ and $\lambda$ respectively.}
 \label{fig:coveringmap}
\end{figure}

\paragraph{} Since a rational tangle $R(p/q)$ is homeomorphic to the tangle $R(1/0)$, its branched double cover is a solid torus. If we consider branched double cover of $R(1/0)$ given in Figure~\ref{fig:coveringmap}. Orienting the curves $\tilde{\mu}$ and $\tilde{\lambda}$ so that $\tilde{\mu}.\tilde{\lambda}=-1$, then classes $[\tilde{\mu}],[\tilde{\lambda}]$ form a basis for $H_1(T^2)$ and with respect to this basis, the class which represents a meridian in the branched double cover of $R(p/q)$ is $p[\tilde{\mu}]+q[\tilde{\lambda}]$. This could be used to give an alternative definition of the slope of a rational tangle.

\subsection{Rational tangles in alternating diagrams}
This identification of rational tangles with the rational numbers depends on the marking on the boundary of the tangle. When a rational tangle occurs as a sub-tangle of an alternating diagram, we can use this to determine the slope. If the rational tangle $T$ is contained in an alternating diagram $D$ then, as in Section~\ref{sec:altdiagrams}, we can colour $D$ so that every crossing has incidence number -1. We then choose a marking on the boundary so that the arc $\lambda$ lies in a shaded region and $\mu$ does not. This allows four choices of markings on the sphere. However, in each case the lifts of $\mu$ and $\lambda$ give the same basis for $H_1(T^2)$, so the slope of $T$ is independent of this choice.
 \begin{figure}[h]
  \centering
  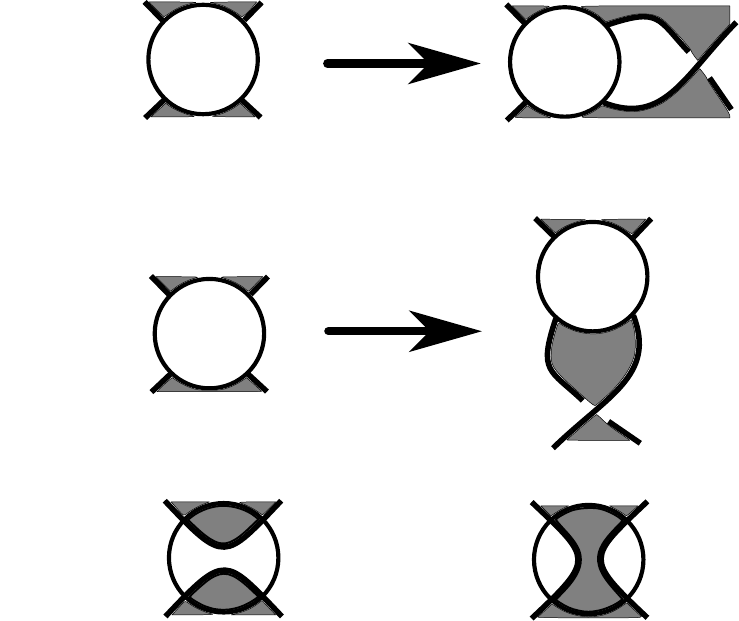
 \caption{How the shading of a rational tangle in an alternating diagram behaves on the tangles $R(0/1)$, $R(1/0)$ and how it is altered by the tangle-building operations $h$ and $v$.}
 \label{fig:shadedhandv}
\end{figure}
We will use the following proposition to find rational tangles in an alternating diagrams and calculate their slopes.

\begin{prop}\label{prop:tangledetection}
Let $D$ be a reduced alternating diagram. Suppose that there is a disk in the plane whose boundary intersects two white regions of $D$.  Let $T$ be the tangle contained in the disk and let $\Gamma_T$ be the subgraph of $\Gamma_D$ induced by $T$. Suppose that $\Gamma_T$ consists of vertices $v_0, \dotsc, v_{l+1}$ for $l\geq 0$, where $v_0$ and $v_{l+1}$ are the white regions which intersect the disk boundary. If there is precisely one edge between $v_i$ and $v_{i+1}$ for $0\leq i< l$ and every remaining edge in $\Gamma_T$ is incident to $v_{l+1}$, then $T$ is a rational tangle. Moreover, if $b_i$ denotes the number of edges incident to $v_i$ in $\Gamma_T$ then the slope of $T$ is $\frac{p}{q}$ where
\begin{equation}\label{eq:tangleslope}
\frac{q}{p} = [b_0, \dotsc, b_l]^-.
\end{equation}
Here $[b_0, \dotsc, b_l]^-$, denotes the continued fraction
$$[b_0, \dotsc, b_l]^-
    = b_0 - \cfrac{1}{b_1
           - \cfrac{1}{\ddots
           - \cfrac{1}{b_l} } }.
$$
\end{prop}
\begin{proof}
We proceed by induction on the number of crossings. If $T$ does not contain any crossings, then it is equivalent as a marked tangle to $R(1/0)$. It is not $R(0/1)$ as we are presuming $\Gamma_T$ has at least two white regions. In this case $\Gamma_T$ consists of two vertices and no edges between them. Therefore $b_0=0$ and \eqref{eq:tangleslope} is clearly satisfied. The conditions in the proposition are such that we can see that $T$ is obtained from a smaller tangle $T'$ by applying one of the operations $h$ or $v$.
\paragraph{} If $b_0>1$, then there is a crossing between $v_0$ and $v_{l+1}$. So we see that $T$ is obtained by applying the operation $v$ to a tangle $T'$ for which the white graph $\Gamma_{T'}$ is obtained by deleting an edge between $v_0$ and $v_{l+1}$. By the inductive hypothesis, we can assume that $T'$ is a rational tangle of slope $\frac{p'}{q'}$, where
$$\frac{q'}{p'} = [b_0-1, \dotsc, b_l]^-.$$
From \eqref{eq:standardtangle}, it follows that $vR(\frac{p'}{q'})=R(\frac{p'}{p'+q'})$. Therefore $T$ is rational of slope $\frac{p}{q}=\frac{p'}{p'+q'}$. Therefore
$$[b_0, \dotsc, b_l]^-=1+\frac{q'}{p'}=\frac{p'+q'}{p'}=\frac{q}{p},$$
as required.
\paragraph{} If $b_0=1$, then there is no crossing between $v_0$ and $v_{l+1}$ and the sole crossing incident to $v_0$ is between $v_0$ and $v_1$. Thus we see that $T$ is obtained by applying the operation $h$ to a tangle $T'$ for which the white graph $\Gamma_{T'}$ is obtained by deleting $v_0$. By the inductive hypothesis, we can assume that $T'$ is a rational tangle of slope $p'/q'$, where
$$\frac{q'}{p'} = [b_1-1, \dotsc, b_l]^-.$$
From \eqref{eq:standardtangle} it follows that $hR(\frac{p'}{q'})=R(\frac{p'+q'}{q'})$ and hence that $T$ has slope $\frac{p}{q}=\frac{p'+q'}{q'}$. Therefore
$$[b_0, \dotsc, b_l]^-=1-\frac{1}{\frac{q'}{p'}+1}=\frac{q'}{p'+q'}=\frac{q}{p},$$
as required.
\end{proof}

\subsection{Tangle replacement and surgery}
Now we suppose that we have a knot or link $K$ with a diagram $D$ obtained by replacing a $1/0$-tangle in a diagram of the unknot $D'$ by a rational tangle of slope $p/q$. The double cover of $S^3$ branched along $D'$ is again $S^3$ and the $1/0$-tangle lifts to give a solid torus $T \subset S^3$. Let $\kappa$ be the knot given by the core of $T$. Let $\lambda_0$ be a null-homologous longitude of $\kappa$ lying in $\partial T$ and $\mu_0\in \mathbb{Z}$ be such that
$$[\tilde\lambda]=\mu_0 [\tilde\mu] + [\lambda_0].$$
If we consider the branched double cover of $D$, we see that it is obtained by cutting out the interior of $T$ and gluing in a solid torus in such a way that $p[\tilde{\mu}]+q[\tilde{\lambda}]$, bounds a disk in the torus. Therefore, we see that $\Sigma(K)$ is obtained by surgery on $\kappa$. In particular,
\begin{equation}\label{eq:surgerycoef}
\Sigma(K)=S^3_{-(\mu_0 +\frac{p}{q})}(\kappa).
\end{equation}
\begin{rem}
There are different conventions for labeling the slope of a rational tangle and the slope of Dehn surgery (c.f the remark after Corollary 4.4 in \cite{Gordon09dehnsurgery}). We have chosen to orient $\tilde{\mu}$ and $\tilde{\lambda}$, so that $\tilde{\mu}.\tilde{\lambda}=-1$. In order to match the usual conventions for Dehn surgery, we would need to reverse the orientation on $\tilde{\lambda}$.
This explains the minus sign appearing in the surgery coefficient of \eqref{eq:surgerycoef}.
\end{rem}

\paragraph{} This construction giving a surgery from a tangle replacement is frequently referred to as the Montesinos trick. In order to determine the sign of the integer $\mu_0$, we will quote a special case for which it is known. Together with the observations of the previous paragraph this will allow us to prove the implication $(iii)\Rightarrow (i)$ in Theorem~\ref{thm:rationalsurgery}.

\begin{prop}[Proof of Theorem~8.1,\cite{ozsvath2005knots}]\label{prop:Montesinos}
Suppose that a knot $K\subset S^3$ can be unknotted by changing a negative crossing to a positive one. Then there is a knot $\kappa \subset S^3$ such that $\Sigma(K)=S_{-\delta/2}^3(\kappa)$ where $\delta=(-1)^{\sigma(K)/2}\det(K)$.
\end{prop}

\begin{prop}\label{prop:3implies1}
Let $K$ be a link with an alternating diagram $D$ containing a rational tangle of slope $\frac{r}{s}\geq 0$. Let $D'$ be the alternating diagram obtained by replacing this tangle with a single crossing $c$. If $c$ is an unknotting crossing in $D'$, which is positive if $\sigma(D')=-2$ and negative if $\sigma(D')=0$, then there is $\kappa \subset S^3$ such that $\Sigma(K)=S_{-p/q}^3(\kappa)$, where $q=r+s$ and $\det K=p=qn-r>0$ for some integer $n$.
\end{prop}
\begin{proof}
Since $c$ is an unknotting crossing, we see that $D$ is obtained by replacing a $-1$-tangle by a $\frac{r}{s}$-tangle in a diagram of the unknot. As shown in Figure~\ref{fig:crossingreplace}, this is equivalent to replacing a $1/0$-tangle by a $\frac{r}{r+s}$-tangle.
\begin{figure}
  \centering
  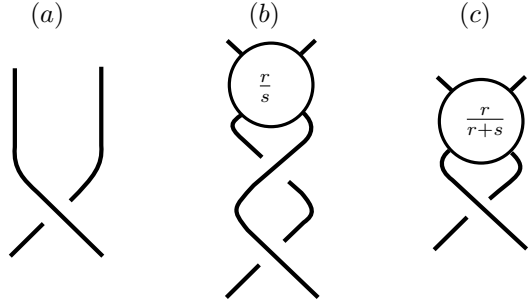
 \caption{Replacing the tangle with $R(-1/1)$ with $R(r/s)$ is equivalent to replacing the tangle $(a)$ with $(b)$. Since $(b)$ is isotopic to $(c)$, this shows that the tangle replacement is equivalent to replacing $R(1/0)$ with $R(\frac{r}{r+s})$.}
 \label{fig:crossingreplace}
\end{figure}
Applying \eqref{eq:surgerycoef} and the discussion preceding it, we see taking the branched double cover over the unknot obtained by changing $c$, gives $\kappa \subset S^3$ and integer $\mu_0 \in \mathbb{Z}$ such that
$$\Sigma(D')=S_{-(\mu_0 + \frac{1}{2})}^3(\kappa)\quad \text{and} \quad \Sigma(K)=S_{-(\mu_0 +\frac{r}{q})}^3(\kappa),$$
where $q=r+s$. The sign conditions on the unknotting crossing $c$ combined with Proposition~\ref{prop:Montesinos} imply that $\mu_0\geq 0$. Taking $n=\mu_0+1$ and $p=nq -r$, we see that
$$\Sigma(K)=S_{-p/q}^3(\kappa).$$
Observe that
$$|H_1(\Sigma(K))|=\det K = |H_1(S_{-p/q}^3(\kappa))|=p$$
to complete the proof.
\end{proof}
Observe that when we consider a tangle of slope $\frac{q-r}{r}$ in Proposition~\ref{prop:3implies1}, we get the implication $(iii) \Rightarrow (i)$ in Theorem~\ref{thm:rationalsurgery}.

\section{Changemaker lattices}\label{sec:cmlattices}
In this section, we define $p/q$-changemaker lattices and explore their properties. Changemaker lattices corresponding to the case $q=1$ were defined by Greene in his solution to the Lens space realization problem \cite{GreeneLRP} and work on the cabling conjecture \cite{greene2010space}. The case $q=2$ arose in his work on unknotting numbers \cite{Greene3Braid}. The more general definition we state here is the one which arises in Gibbons' work \cite{gibbons2013deficiency}.

\begin{defn}We say $(\sigma_1, \dotsb, \sigma_t)$ satisfies the {\em changemaker condition}, if the following conditions hold,
$$0\leq \sigma_1 \leq 1, \text{ and } \sigma_{i-1} \leq \sigma_i \leq \sigma_1 + \dotsb + \sigma_{i-1} +1,\text{ for } 1<i\leq t.$$
\end{defn}

\begin{defn}\label{defn:CMlattice}
First suppose that $q=1$, so that $p/q>0$ is an integer. Let $f_0, \dotsc, f_t$ be an orthonormal basis for $\mathbb{Z}^t$. Let $w_0=\sigma_1 f_1 + \dotsb + \sigma_t f_t$ be a vector such that $\norm{w_0}=p$ and $(\sigma_1, \dotsb, \sigma_t)$ satisfies the changemaker condition, then
$$L=\langle w_0\rangle^\bot \subseteq \mathbb{Z}^{t+1}$$
is a {\em$p/q$-changemaker lattice}.

\paragraph{}Now suppose that $q\geq 2$ so that $p/q>0$ is not an integer. Write $p/q$ in the form $n-r/q$, where $0<r<q$. This has continued fraction expansion of the form, $p/q=[a_0,a_1, \dotsb , a_l]^{-}$, where $a_k\geq 2$ for $1\leq k \leq l$, $a_0=n\geq 1$, and $q/r=[a_1, \dotsb , a_l]^{-}$. Now define $m_0=0$ and $m_k=\sum_{i=1}^ka_i -k$ for $1\leq k \leq l$. Set $s=m_{l}$ and let $f_1, \dotsc, f_t, e_0, \dotsc, e_s$ be an orthonormal basis for the lattice $\mathbb{Z}^{t+s+1}$.
Let $w_0=e_0+\sigma_1 f_1 + \dotsb + \sigma_t f_t,$ be a vector such that $\{\sigma_1, \dotsb, \sigma_t\}$ satisfies the changemaker condition and $\norm{w_0}=n$. For $1\leq k \leq l$, define
$$w_k=-e_{m_{k-1}}+e_{m_{k-1}+1}+ \dotsb + e_{m_{k}}.$$
We say that $L=\langle w_0, \dotsc, w_l\rangle^\bot \subseteq \mathbb{Z}^{t+s+1}$ is a {\em$p/q$-changemaker lattice}.
\end{defn}

\paragraph{} We include the definition in the case where $q=1$ for completeness, however throughout this paper we will be concerned with the case where $q\geq 2$. Observe that by construction for a changemaker lattice $L=\langle w_0, \dotsc, w_l\rangle^\bot \subseteq \mathbb{Z}^{t+s+1}$, we have
\[
w_i.w_j =
  \begin{cases}
   a_j            & \text{if } i=j \\
   -1       & \text{if } |i-j|=1\\
   0        & \text{otherwise.}
  \end{cases}
\]

\begin{example}\label{example:3over5}
Consider $p/q=n-3/5$, the continued fraction of this is $p/q=[n,2,3]^-$, so a $p/q$-changemaker lattice takes the form
$$L=\langle \sigma, -e_0+e_1,-e_1+e_2+e_3 \rangle^\bot \subset \mathbb{Z}^{r+4}=\langle e_0, e_1, e_2, e_3, f_1, \dotsc, f_r \rangle,$$
where $\sigma=e_0 + \sigma_1 f_1 + \dotsb + \sigma_r f_r$ and $(\sigma_1, \dotsc, \sigma_r)$ satisfies the changemaker condition.
\end{example}
If $p/q< 1$, then $w_0=e_0$ and $x\cdot e_0=0$ for all $x\in L$. We will ignore this degenerate case and assume from now on that $p/q>1$.
The term changemaker is explained by the following combinatorial proposition.
\begin{prop}[Brown \cite{brown1961note}]\label{prop:CMcondition}
Let  $\sigma =\{\sigma_1, \dotsc , \sigma_s\}$, with $\sigma_1\leq \dotsb \leq \sigma_s$. There is $A\subseteq \{ 1, \dotsc , s\}$ such that $k=\sum_{i\in A} \sigma_i$, for every integer  $k$ with $0\leq k\leq \sigma_1 + \dotsb + \sigma_s$, if and only if $\sigma$ satisfies the changemaker condition.
\end{prop}

\subsection{Fractional parts}\label{sec:fracparts}
For the duration of this section
$$L=\langle w_0, \dotsc, w_l\rangle^\bot \subseteq \mathbb{Z}^{t+s+1}=\langle f_1, \dotsc , f_t, e_0, \dotsc e_s \rangle$$ will be a $p/q$-changemaker lattice, with $p>q\geq 2$ and $w_0=e_0+\sigma_1 f_1 + \dotsb + \sigma_t f_t$.
The {\em fractional part} $L_F$ is defined to be
$$L_F=\langle w_1,\dotsc, w_l \rangle^\bot \cap \langle e_0, \dotsc , e_s \rangle,$$
and the {\em integer part} is
$$L_I=\langle w_0 \rangle^\bot \cap \langle e_0, f_1, \dotsc , f_t \rangle.$$

For any $x\in L$, we will use the notation $x_0 = x\cdot e_0$. Observe that $L_F\cap L_I = \langle e_0 \rangle$. So given $x \in L$ there are $x_I\in L_I$ and $x_F \in L_F$, such that $x=x_I + x_F - x_0e_0$. It is a straightforward calculation that for any $x,y \in L$,
\begin{equation}\label{eq:splittingproduct}
x\cdot y = x_I \cdot y_I - x_0 y_0 + x_F \cdot y_F.
\end{equation}

\paragraph{} In order to study $L_F$, we will construct a basis for it in the following way.
Let $M=\{0, \dotsc, s \} \setminus \{m_0, \dotsc , m_{l}\}$, where $m_0, \dotsc , m_{l}$ are as in Definition~\ref{defn:CMlattice}.
For each $k \in M$ with $k<\max M$, let $k'$ be minimal in $M$ with $k'>k$. Using this define
$$v_k'=-e_k + e_{k+1} + \dotsb + e_{k'},$$
for each $k\in M$.
We define $v_0= e_0+ \dotsb + e_{\min M}$. Finally we reindex the $v_k'$ to get $v_0, \dotsb, v_{m}$, where $m=|M|-1$, satisfying
\[
v_i\cdot v_j =
  \begin{cases}
   \norm{v_i}            & \text{if } i=j \\
   -1       & \text{if } |i-j|=1\\
   0        & \text{if } |i-j|>1.
  \end{cases}
\]
These form a basis for $L_F$. For $1\leq i \leq m$, we also have $v_i \cdot w_0 =0$, so $v_i \in L$.

\paragraph{} In the case of Example~\ref{example:3over5}, the fractional part is
$$L_F= \langle -e_0+e_1,-e_1+e_2+e_3\rangle^\bot \subseteq \langle e_0, \dotsc, e_3\rangle.$$
And the basis for this as constructed in this section consists of $v_0=e_0+e_1+e_2$ and $v_1=-e_2+e_3$.

\subsection{Irreducibility}
Now we wish to study the irreducibility of certain vectors in $L$ and $L_F$. Recall that $z \in L$ is irreducible if it cannot be written in the form $z=x+y$ for non-zero $x$ and $y$ in $L$ with $x\cdot y\geq 0$.
\begin{lem}\label{lem:generalirred}
Suppose that $\sigma_i\geq 1$ for $1\leq i \leq t$ and that $z\in L$ takes the form
$$z=-f_k + \sum_{i \in A} f_i + \sum_{j\in B} e_j,$$
for subsets $A\subseteq \{1, \dotsc, k-1\}$ and $B \subseteq \{0, \dotsc, s\}$. Then $z$ is irreducible.
\end{lem}
\begin{proof}
Suppose we have $z$ as in the statement of the lemma, and that we may write $z=x+y$ with $x\cdot y \geq 0$. This gives
\begin{equation}\label{eq:irredsum}
x \cdot y = \sum_{i=1}^t (x\cdot f_i) (y\cdot f_i) + \sum_{j=0}^s (x\cdot e_j) (y\cdot e_j)\geq 0.
\end{equation}
For all $i$ and $j$, we have $x\cdot f_i+y\cdot f_i = z\cdot f_i \in \{-1,0,1\}$ and $x\cdot e_j+y\cdot e_j = z\cdot e_j \in \{-1,0,1\}$. This implies that $(x\cdot f_i) (y\cdot f_i)\leq 0$ and $(x\cdot e_j) (y\cdot e_j)\leq 0$ for all $i$ and $j$. In particular, this shows that every summand in \eqref{eq:irredsum} must be 0.

It follows that $(x\cdot f_k,y\cdot f_k)\in \{((-1, 0), (0, -1)\}$ and $(x\cdot f_i,y\cdot f_i),(x\cdot e_j,y\cdot e_j)\in \{(0,0),(1, 0), (0,1)\}$ for all $i\ne k$ and all $j$. We may assume $x\cdot f_k=-1$, which implies $y\cdot f_i, y\cdot e_j\geq 0$ for all $i$ and $j$. Since we are assuming $\sigma_i\geq 1$ for all $i$, the only vector $w\in L$ with the property that $w\cdot f_i, w\cdot e_j\geq 0$ for all $i$ and $j$ is the vector $w=0$. Therefore, $y=0$ and $z$ is irreducible.
\end{proof}

\begin{lem}\label{lem:fracirredcondition}
The vector $x_F \in L_F$ is irreducible if and only if $x_F$ is in the form $x_F=\pm (v_a + \dotsb + v_b)$, for some $0\leq a \leq b \leq m$.
\end{lem}
\begin{proof}
Observe that we may consider $L_F$ as a graph lattice $\Lambda(G)$ of a connected graph with vertices $v_0, \dotsc , v_m , w= -(v_0+ \dotsb + v_m)$. The vertices $v_0, \dotsc, v_m$ span a path in $G$ and both the end points of this path, $v_0$ and $v_m$ are connected to the vertex $w$. Now we may appeal to Lemma~\ref{lem:irreducible}. Let $G_1$ be a non-empty connected subgraph with connected complement. We may assume that $w\notin G_1$. Thus $G_1$ is connected subgraph of the path spanned by $v_0, \dotsc, v_m$. This implies there are $a\leq b$, such that $G_1$ has vertex set $v_a, \dotsc , v_b$. One can see that any such $G_1$ also has a connected complement.
\end{proof}
\begin{lem}\label{lem:fracpartirred}
If $x \in L$ is irreducible, then $x_F\in L_F$ is also irreducible.
\end{lem}
\begin{proof}
Let $x_0=x\cdot e_0$. We may write $x=x_F+x_I - x_0e_0$, with $x_I\in L_I$ and $x_F \in L_F$. If $x_0=0$, then $x_I,x_F \in L$ and $x_I \cdot x_F=0$. Since $x=x_I + x_F$, irreducibility of $x$ implies that $x_F=x$ or $x_F=0$.
\paragraph{} Now we suppose $x_0\ne 0$. We may assume that $x_0>0$ and that $x_F=\sum_{i=1}^{m}c_iv_i$, where $c_0=x_0$. Let $g_F>0$ be minimal such that $c_{g_F+1}\leq 0$. For convenience, we will take $c_{m+1}=0$. Now consider
$$z_F=v_0+ \dotsb + v_{g_F} \in L_F.$$
By Lemma~\ref{lem:fracirredcondition}, this is irreducible. We shall prove the lemma by showing $x_F=z_F$. First we need to bound the quantity $(x_F - z_F)\cdot z_F$. We have
\begin{align*}
(x_F - z_F)\cdot z_F &= \sum_{i=0}^{g_F}v_i \cdot (x_F-z_F) \\
    &= (\norm{v_0}-1)(c_0-1) + \sum_{i=1}^{g_F-1}(\norm{v_i}-2)(c_i -1) + (\norm{v_{g_F}}-1)(c_{g_F}-1) - c_{g_F+1}\\
    &\geq (\norm{v_0}-1)(c_0-1)\\
    &=(\norm{v_0}-1)(x_0-1).
\end{align*}
In particular, this shows
\begin{equation}\label{eq:bound1}
(x_F - z_F)\cdot z_F \geq x_0-1.
\end{equation}
Now let $g_I$ be minimal such that $x\cdot f_{g_I}\leq 0$. By Proposition~\ref{prop:CMcondition}, there is $A \subseteq \{1, \dotsc, g_I - 1\}$ such that $\sigma_{g_I} - 1 =\sum_{i \in A} \sigma_i$. Hence, if we define
$$z_I=-f_{g_I} + e_0 + \sum_{i \in A}f_i,$$
then $z_I \in L_I$. Since
\begin{equation*}
(x_I-z_I)\cdot z_I = -x\cdot e_{g_I} - 1 + \sum_{i\in A}(x\cdot e_i -1) + (x_0-1),
\end{equation*}
we have the bound
\begin{equation}\label{eq:bound2}
(x_I-z_I)\cdot z_I\geq x_0 -2.
\end{equation}
Now consider
$$z=z_I+z_F-e_0\in L.$$
Using \eqref{eq:splittingproduct} along with the inequalities \eqref{eq:bound1} and \eqref{eq:bound2}, we have
\begin{align*}
(x-z)\cdot z &= (x_I-z_I)\cdot z_I - (x_0-1) + (x_F - z_F)\cdot z_F \\
    &\geq x_0-2 \geq -1.
\end{align*}
If $(x-z)\cdot z\geq 0$, then the irreducibility of $x$ implies that $x=z$. Otherwise the above inequality shows that $(x-z)\cdot z=-1$ and in particular that $x_0=1$. Thus we may consider $z'=x_I+z_F-e_0 \in L$. Since
$$(x-z')\cdot z'=(x_F-z_f)\cdot z_F \geq x_0 -1 =0,$$
it follows that $x=z'$. In either case $x_F=z_F$ which is irreducible.
\end{proof}

\subsection{Indecomposability}
Now we study the indecomposability of $p/q$-changemaker lattices. Recall that the lattice $L$ is indecomposable if it cannot be written as an orthogonal direct sum $L=L_1 \oplus L_2$ with $L_1,L_2 \ne 0$.
\begin{lem}\label{lem:indecomp}
Let $L$ be a $p/q$-changemaker lattice for $q\geq 2$. The following are equivalent:
\begin{enumerate}[(i)]
    \item $L$ is indecomposable.
    \item $\sigma_i\geq 1$ for $t\geq i\geq 1$
    \item $L$ contains no vectors of norm 1.
\end{enumerate}
\end{lem}
\begin{proof}
This is the natural generalisation of \cite[Lemma~2.10]{mccoy2013alternating}. Since $f_i\cdot w_0 =\sigma_i$, it is clear that $f_i \in L$ if and only if $\sigma_i =0$. For any $e_i$, there is always some $w_j$, for which $w_j \cdot e_i =1$. This shows $(ii) \Leftrightarrow (iii)$. Furthermore it is clear that if $\sigma_i=0$ for some $i$, then $L \cong \frac{L}{\langle f_i \rangle}\oplus \langle f_i \rangle$. This gives the implication $(i) \Rightarrow (ii)$.
\paragraph{}Now we prove $(ii)\Rightarrow (i)$. For each $1\leq k\leq t$, Proposition~\ref{prop:CMcondition} shows that either
$$\sigma_k=1+\sigma_1 + \dotsb + \sigma_{k-1},$$
or there is $A_k \subseteq\{1,\dotsc, k-1 \}$ such that
$$\sigma_k = \sum_{i\in A_k} \sigma_i.$$
In the first case, define
$$u_k=-f_k+f_{k-1}+ \dotsb + f_1 + v_0 \in L,$$
and in the second, we may define
$$u_k =-f_k + \sum_{i\in A_k} f_i \in L.$$
Observe that the set $B=\{u_1, \dotsc, u_t, v_1, \dotsc v_m\}$ forms a basis for $L$. By Lemma~\ref{lem:fracirredcondition} and Lemma~\ref{lem:generalirred}, every element of $B$ is irreducible. Suppose that we can write $L=L_1\oplus L_2$. Observe that if $x \notin L_1 \cup L_2$, then it can be written as $x=x_1+x_2$, with $x_i\in L_i \setminus \{0\}$ and $x_1\cdot x_2=0$. In particular, $x$ is reducible. Therefore, we must have $B\subseteq L_1\cup L_2$.

Without loss of generality, suppose $u_1=-f_1+v_0 \in L_1$. We will show that this implies $B \subseteq L_1$. Since $v_1\cdot u_1=-1$ and $v_i\cdot v_{i+1}=-1$ for all $1\leq i <m$, it follows that $v_i \in L_1$ for all $i$. Consider now $u_k$ for some $k>1$. First, suppose that $u_k=-f_k+f_{k-1}+ \dotsb + f_1 + v_0$. In this case, $u_1 \cdot u_k= \norm{v_0}-1\geq1$, which implies $u_k \in L_1$. Now suppose that $u_k =-f_k + \sum_{i\in A_k} f_i$. Let $l=\min A_k$, so that $u_k\cdot u_l=-1$. This allows us to prove inductively that $u_k \in L_1$ for all $k$. Therefore we have $B\subseteq L_1$ and hence $L=L_1$. This proves that $L$ is indecomposable, which completes the proof.
\end{proof}

\section{Alternating diagrams and changemaker lattices}\label{sec:main}
The objective of this section is to show that if $D$ is a reduced alternating diagram with
$$\Lambda_D \cong \langle w_0, \dotsc, w_l\rangle^\bot \subseteq \mathbb{Z}^{r+s+1},$$
where $L=\langle w_0, \dotsc, w_l\rangle^\bot$ is a $p/q$-changemaker lattice with $p>q\geq 2$, then there is a sequence of flypes to obtain an alternating diagram which can be obtained by rational tangle replacement from an almost-alternating diagram of the unknot. Fix a choice of isomorphism,
$$\iota_D: \Lambda_D \longrightarrow L.$$
This gives a distinguished collection of vectors in $L$ given by the image of the vertices of $\Gamma_D$. We call this collection $V_D$, and in an abuse of notation we will fail to distinguish between a vertex of $\Gamma_D$ and the corresponding element in $V_D$.

Since $D$ is reduced, $\Gamma_D$ contains no self loops or cut edges. In particular, there are no vectors of norm 1 in $\Lambda_D$, so Lemma~\ref{lem:indecomp} implies that $\Lambda_D$ is indecomposable. By Lemma~\ref{lem:2connectgraphlat}, $\Gamma_D$ is 2-connected and any $v \in V_D$ is irreducible.

It will be necessary for us to flype $D$ to obtain a new reduced alternating diagrams. In all cases, this flype will be an application of Lemma~\ref{lem:flype1} or a flype as appearing in Figure~\ref{fig:flype2}. In either case, if $D'$ is the diagram we obtain from such a flype, then we get a natural choice of $V_{D'}\subset L$ and hence an isomorphism
$$\iota_{D'}: \Lambda_{D'} \rightarrow L.$$
Whenever we flype, we will implicitly use these choices of isomorphism to speak of $V_{D'}$ without ambiguity.

\subsection{The half-integer case}
Alternating diagrams for which the Goeritz lattice is isomorphic to a $p/2$-changemaker lattice were studied in the proof of Theorem~\ref{thm:unknotting}. In this case, we have a diagram with Goeritz form $\Lambda_D$ with
$$\Lambda_D\cong L= \langle w_0, e_1-e_0\rangle^\bot\subseteq \langle f_1, \dotsc , f_t, e_0, e_1 \rangle = \mathbb{Z}^{t+2}.$$
For every $x\in L$ we have $z\cdot e_0 = z\cdot e_1$ and there are precisely two vertices in $u \in V_D$ with $u\cdot e_0 \ne 0$. These can be written in the form $u_1+e_0+e_1$ and $u_2-e_0-e_1$, where $u_1\cdot e_0=u_2\cdot e_0=0$. Any crossing between the regions corresponding to these two vectors is called a {\em marked crossing}. For example, see Figure~\ref{fig:markedcrossing}.
\begin{figure}[h]
  \centering
  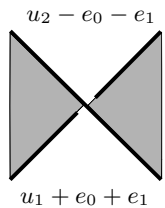
 \caption{A marked crossing.}
 \label{fig:markedcrossing}
\end{figure}
Any marked crossing is an unknotting crossing. The following is an amalgamation of results \cite[Theorem~3 and Theorem~4]{mccoy2013alternating}.
\begin{thm}\label{thm:markedmeansunknotting}
Let $D$ be a reduced alternating diagram of $K$, and suppose that the lattice $\Lambda_D$ is isomorphic to a $p/2$-changemaker lattice. Then there is at least one marked crossing in $D$, and any marked crossing is also an unknotting crossing, which is negative if $\sigma(K)=0$ and positive if $\sigma(K)=-2$.
\end{thm}

\subsection{The fractional tangle}
In order to prove Theorem~\ref{thm:rationalsurgery}, we will study the vertices in $V_D$ which have a non-zero fractional part. This will allow us to obtain a diagram in which $L_F$ specifies a tangle. We will take $v_0, \dotsc, v_m$ to be the basis of $L_F$ as constructed in Section~\ref{sec:fracparts}.
First we show that there is a sequence of flypes to a diagram in which $v_1,\dotsc , v_m$ are vertices.

\begin{lem}\label{lem:flypetorational}
We may flype to obtain a diagram in which $v_1,\dotsc , v_m$ are vertices.
\end{lem}
\begin{proof}
Let $1\leq c\leq m$ be maximal such that $v_c$ is not a vertex. By Lemma~\ref{lem:fracirredcondition}, $v_c$ is irreducible. Therefore, Lemma~\ref{lem:irreducible} shows that there is $R\subseteq V_D$ such that $v_c=\sum_{x\in R} x$. In particular, this implies that there exists some $u \in V_D$ with $u_F=v_a+ \dotsb +v_b$, for $a\leq c \leq b$. As $v_c$ is assumed not to be a vertex, $v_c\ne u$ and hence $a<b$.  Since
$$v_b \cdot u = v_b \cdot u_F = \norm{v_b} -1>0,$$
it follows that $v_b$ cannot be a vertex. This implies $b=c$, so that $u_F$ takes the form $u_F=v_a+\dotsb + v_c$, for $a<c$.

\paragraph{} Since $(u-v_c)\cdot v_c=v_{c-1}\cdot v_c =-1$, we may apply a flype as given by Lemma~\ref{lem:flype1} to obtain a new diagram $D'$ with $v_c\in V_{D'}$. Since $v_c \cdot v_k \leq 0$ for all $k\ne c$, it follows that $v_{c+1}, \dotsc, v_m$ are still vertices. Iterating this process shows that we may flype to obtain a diagram $\widetilde{D}$ with $v_1,\dotsc , v_m$ in $V_{\widetilde{D}}$.
\end{proof}
 From now on we will assume that we have performed flypes as in Lemma~\ref{lem:flypetorational}, and have a diagram $D$ with $v_1,\dotsc , v_m\in V_D$.
 \begin{lem}\label{lem:markersunique}
 There is a unique vertex $v$ with $v\cdot e_0>0$ and this has $v_F=v_0$. There is a unique vertex $w$ with $w\cdot e_0<0$ and this has $w_F = -(v_0 + \dotsb + v_m)$.
 \end{lem}
 \begin{proof}
 First we will show that any vertex $v\in V_D$ with $v\cdot e_0>0$ necessarily has $v_F=v_0$. Since $v_F$ is irreducible, it is of the form $v_F=v_0+\dotsb + v_b$. If $b>0$, then $v_b \in V_D$ and $v_b \cdot v >0$ is a contradiction. Thus, $v_F=v_0$. Similarly we suppose that $w$ is any vertex with $w\cdot e_0 <0$. By irreducibility of $w_F$, it follows that $w_F= -(v_0 + \dots + v_c)$. If $c<m$ then we have $v_{c+1}\in V_D$ and we get a contradiction from $w\cdot v_{c+1}=1>0$. This proves the statement about the fractional part of $w$.
 \paragraph{} Now we show the uniqueness of $v$. Since $\sum_{x \in V_D} x\cdot e_0 =0$, this is also implies the uniqueness of $w$. Suppose there are vertices $u_1$ and $u_2$, with $(u_1)_{F}=(u_2)_{F} =v_0$. Consider $U=u_1 + u_2$. This has a $U_F=2v_0$. Now let $g$ be minimal such that $U\cdot f_g \leq 0$. By Proposition~\ref{prop:CMcondition}, there is $A\subseteq \{1, \dotsc , g-1 \}$, such that $\sigma_g-1 = \sum_{i \in A}\sigma_i$. Hence we may take $z=-f_g + v_0 + \sum_{i \in A} f_i \in L$.
 Now consider the inequality,
 \begin{align*}
 (U-z)\cdot z &= -V\cdot f_g -1 + \sum_{i \in A}(V\cdot f_i -1) + \norm{v_0} \\
    &\geq -1 + \norm{v_0}>0.
 \end{align*}
 Since this exceeds the bound in Lemma~\ref{lem:usefulbound}, it follows that $U$ is not the sum of distinct vertices. In particular this implies $u_1=u_2$, which gives the required uniqueness statement.
 \end{proof}

 Now let $v$ and $w$ be the vertices as determined by Lemma~\ref{lem:markersunique}. We wish to determine the number of edges between $v$ and $w$. These edges along with the vertices $v_1,\dotsc, v_m$, will provide the rational tangle we are seeking.

\begin{lem}\label{lem:coefbound}
Suppose $x = \sum_{u\in R}u \in L$ for some $R\subseteq V_D$, then $|x\cdot f_1|\leq 2$. Furthermore, if $x$ is irreducible with $x\cdot e_0\ne 0$, then $|x\cdot f_1|\leq 1$.
\end{lem}
\begin{proof}
We may assume $x\cdot f_1\geq 0$. Let $g>1$ be minimal such that $x\cdot f_g\leq 0$. By Proposition~\ref{prop:CMcondition}, we may write $\sigma_g -1 = \sum_{i \in A} \sigma_i$ for some $A \subseteq \{1, \dotsc , g-1\}$. So we have $z=-f_g + f_1 + \sum_{i \in A} f_i\in L$. Let $z\cdot f_1 =\epsilon$ and observe that $\epsilon \in \{1,2\}$. By Lemma~\ref{lem:usefulbound}, $(x-z)\cdot z\leq 0$. This gives
\begin{align*}
0\geq (x-z)\cdot z&= -x\cdot f_g -1 + \sum_{i\in A\setminus \{1\}}(x\cdot f_i -1) + \epsilon(x\cdot f_1-\epsilon)\\
    &\geq -1 + \epsilon(x\cdot f_1-\epsilon).
\end{align*}
Therefore, $x\cdot f_1\leq \epsilon + \frac{1}{\epsilon}\leq \frac{5}{2}$, which gives the required bound.

\paragraph{}Suppose now that $x$ is irreducible, and $x\cdot e_0 > 0$. By Lemma~\ref{lem:fracirredcondition}, this implies that $x\cdot e_0=1$. Let $g>0$ be minimal such that $x\cdot f_g\leq 0$. If $g=1$, then let $z=-f_1 + x_F$. By irreducibility, it follows that either $x=z$ or
$$(x-z)\cdot z= -(x\cdot f_1+1)\geq -1,$$
which implies $x\cdot f_1 =0$.
In either case, $0\geq x\cdot f_1\geq -1$, as required.
Suppose $g>1$, we may write $\sigma_g -1 = \sum_{i \in A} \sigma_i$ for some $A \subseteq \{1, \dotsc , g-1\}$. If $1\in A$, set $z=-f_g + x_F + \sum_{i \in A} f_i$. If $1\notin A$, set $z=-f_g + f_{1} + \sum_{i \in A} f_i$. In either eventuality we get $z\in L$ with $z\cdot f_1=1$. By irreducibility, it follows that either $x=z$ or
$$-1\geq (x-z)\cdot z \geq x\cdot f_1 - 1-(x\cdot f_g+1)\geq x \cdot f_1 -2,$$
which implies $x \cdot f_1 = 1$. Thus we get the necessary bounds on $x\cdot f_1$ in all cases.
\end{proof}

Now we can prove a bound on $v\cdot w$.

\begin{lem}\label{lem:markerbounds}
$v\cdot w \leq v_F \cdot w_F + 1$.
\end{lem}
\begin{proof}
By Lemma~\ref{lem:generalirred}, the vector $z=-f_1 + v_0$ is irreducible, so by Lemma~\ref{lem:irreducible} there is $R\subseteq V_D$, such that $z=\sum_{x \in R}x$. By considering $z\cdot e_0$, it follows that $w\notin R$ and $v\in R$. Thus $z-v+w$ is also a sum of vertices. Applying Lemma~\ref{lem:coefbound} gives
$$(z-v+w)\cdot f_1 = -1 -v\cdot f_1 + w\cdot f_1 \geq -2.$$
This implies $v\cdot f_1 \leq 0$ or $w \cdot f_1 \geq 0$. We will now prove the lemma for the case $v\cdot f_1 \leq 0$. The argument can easily be modified to treat the case $w \cdot f_1 \geq 0$.

\paragraph{}Suppose that $v\cdot f_1 \leq 0$. By Lemma~\ref{lem:coefbound}, this implies $v\cdot f_1 \in \{0,-1\}$. As before, we consider $z=-f_1 + v_0$. If $v\cdot f_1 = -1$, then the irreducibility of $v$ implies $v=z$ and we have
$$v\cdot w=-w\cdot f_1+w_F\cdot v_0\leq 1 + w_F\cdot v_0=1+w_F \cdot v_F,$$
which is the required bound.
If $v \cdot f_1 =0$, then $(v-z)\cdot z =-1$. Since
\begin{align*}
v\cdot w &= (v-z)\cdot w + z\cdot w \\
    & = (v-z)\cdot w - w\cdot f_1 + v_F \cdot w_F,
\end{align*}
we are required to show $(v-z)\cdot w - w\cdot f_1 \leq 1$. However, by Lemma~\ref{lem:cutedge},
$(v-z)\cdot w \leq 1$, and by Lemma~\ref{lem:coefbound}, $w\cdot f_1\geq -1$. Thus it suffices to show that $w\cdot (v-z)\leq 0$ or $w\cdot f_0\geq 0$. Suppose $(v-z)\cdot w = 1$. By Lemma~\ref{lem:cutedge}, this implies the existence of a vertex $u\notin \{w,v\}$ such that $u\cdot z=1$ and $u+w$ is irreducible. The condition $u\cdot z =1$ implies that $u\notin \{v_1,\dotsc , v_m\}$, so such a $u$ satisfies $u_F=0$. Therefore we are required to have $u\cdot f_1=-1$. Using the irreducibility of $u+w$, Lemma~\ref{lem:coefbound} implies
$$(u+w)\cdot f_1 = -1 + w\cdot f_1 \geq -1.$$
This implies that $w \cdot f_1 \geq 0$, which completes the proof for this case.
\paragraph{} If $w \cdot f_1 \geq 0$ one can consider $z=f_1-(v_0 + \dotsb + v_m)$, which allows one to carry out an almost identical argument.
\end{proof}

The vertices $v_1, \dotsc , v_m$ form a path between $v$ and $w$ in $\Gamma_D$ and there are at least $|w_F \cdot v_F| -1$ edges between $v$ and $w$. Performing some flypes of the form given by Figure~\ref{fig:flype2} if necessary, we may assume that there is a disk in the plane whose boundary intersects only the regions $v$ and $w$ and its interior contains precisely $|w_F \cdot v_F| -1$ of the crossings between $v$ and $w$, the regions $v_1, \dotsc, v_m$, and all crossings incident to them. In particular, this means that $v, w, v_1, \dotsc , v_m$ are the regions of a subtangle in $D$, and the white graph $\Gamma_D$ in the interior of the disk is as in Figure~\ref{fig:fractangle}. We will call this tangle a {\em fractional tangle}.

\begin{figure}[h]
  \centering
  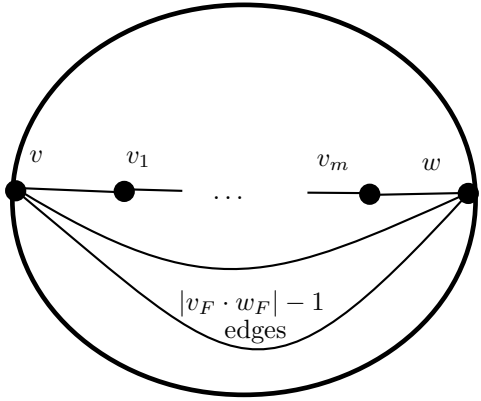
 \caption{The white graph of the fractional tangle.}
 \label{fig:fractangle}
\end{figure}
\begin{lem}\label{lem:fractionalslope}
A fractional tangle is a rational tangle of slope $\frac{q-r}{r}$.
\end{lem}
\begin{proof}
A fractional tangle satisfies the hypotheses of Proposition~\ref{prop:tangledetection}, so it is a rational tangle. It remains to calculate the slope. Suppose the slope is $\frac{\alpha}{\beta}$. By \eqref{eq:tangleslope}, this is given by the continued fraction:
\begin{equation}\label{eq:betaalpha1}
\frac{\beta}{\alpha} =[|w_F \cdot v_F| -1, \norm{v_1}, \dotsc , \norm{v_m}]^-.
\end{equation}
If we write the continued fraction for $p/q$ in the form
\begin{equation}\label{eq:pq1}
\frac{p}{q}=[a_0+1, 2^{[a_1-1]}, a_2 + 2, \dotsc , 2^{[a_{2k-1}-1]}]^-,
\end{equation}
where $2^{[k]}$ denotes $\underbrace{2, \dotsc, 2}_{k}$,
then from the definition of a changemaker lattice, we see that the continued fraction in \eqref{eq:betaalpha1} is
\begin{equation}\label{eq:betaalpha2}
\frac{\beta}{\alpha}=[a_1, 2^{[a_2-1]}, a_3+2, \dotsc , a_{2k-1}+1]^-.
\end{equation}
The following formula allows us to switch between the two types of continued fraction (for example, see \cite{popescu04continued}):
\begin{equation}\label{eq:pointrule}
[c_0,\dotsc, c_{2d-1}]^+=
\begin{cases}
[c_0+1,2^{[c_1-1]},a_2+2, 2^{[a_3-1]}, \dotsc , c_{2d-2}+2, 2^{[c_{2d-1}-1]}]^-, \quad d\geq 2\\
[c_0+1,2^{[c_1-1]}]^- \quad d=1.
\end{cases}
\end{equation}
Applying \eqref{eq:pointrule} to \eqref{eq:pq1} and \eqref{eq:betaalpha2} gives the continued fractions
$$\frac{p}{q}=[a_0,a_1, \dotsc, a_{2k-1}]^+$$
and
$$\frac{\beta}{\alpha}=[a_1-1, a_2, \dotsc, a_{2k-1}]^+.$$
Since we may write $\frac{p}{q}$ in the form $\frac{p}{q}=n-1+\frac{q-r}{q}$,
we have $a_0=n-1$ and
$$\frac{q}{q-r}=[a_1, a_2, \dotsc, a_{2k-1}]^+.$$
It follows that
$$\frac{\beta}{\alpha}= \frac{q}{q-r}-1=\frac{r}{q-r},$$
and hence that the slope of the fractional tangle is $\frac{\alpha}{\beta}=\frac{q-r}{r}$, as required.
\end{proof}

We are now in a position to prove the following proposition.
\begin{prop}\label{prop:CMidentifiestangle}
Let $D$ be an alternating link diagram and suppose that $\Lambda_D$ is isomorphic to a $p/q$-changemaker lattice $L$, where $p/q$ can be written in the form $p/q=n-r/q$. Then there is a sequence of flypes to a diagram $\widetilde{D}$ which contains a fractional tangle. Furthermore, if the fractional tangle is replaced by a single crossing $c$ to obtain a alternating diagram $\widetilde{D}'$, then $\Lambda_{\widetilde{D}'}$ is isomorphic to a $(n-1/2)$-changemaker lattice in such a way that $c$ is a marked crossing.
\end{prop}
\begin{proof}
Suppose that $\Lambda_D$ is isomorphic to the $p/q$-changemaker lattice $L$ defined as in Definition~\ref{defn:CMlattice},
$$L=\langle w_0, \dotsc, w_l\rangle^\bot \subseteq \mathbb{Z}^{t+s+1}.$$
Let $v_0,\dotsc, v_m$, be the basis for $L_F$ as constructed in Section~\ref{sec:fracparts}. By Lemma~\ref{lem:flypetorational} we may flype to a diagram in which $v_1, \dotsc, v_m$ are vertices. In such a diagram, Lemma~\ref{lem:markersunique} shows that there are vertices $v$ and $w$ with $v_F=v_0$ and $w_F=-(v_0+\dotsb + v_m)$. By Lemma~\ref{lem:markerbounds} and discussion following it we may further flype to a diagram $\widetilde{D}$ in which there is a fractional tangle which is the one determined by the regions $v_1, \dotsc, v_m$ and $|v_F\cdot w_F|-1$ crossings between $v$ and $w$. Now let $\widetilde{D}'$ be the alternating diagram obtained by replacing this fractional tangle with a single crossing $c$. Since $v_1, \dotsc, v_m,v$ and $w$ are the only regions in $\widetilde{D}$ with non-zero fractional part, we see that $\Lambda_{\widetilde{D}'}$ admits an embedding into $\langle f_1, \dotsc f_t, e_0, e_1 \rangle \subseteq \mathbb{Z}^{t+2}$, where $V_{\widetilde{D}'}$ is obtained from $V_{\widetilde{D}}$ by deleting $v_1, \dotsc, v_m$ and replacing $v$ and $w$ by $\tilde{v} =v_I+e_1$ and $\tilde{w}=w_I - e_1$, respectively. By considering the image of this embedding we see that $\Lambda_{\widetilde{D}'}$ is isomorphic to the $(n-1/2)$-changemaker lattice
$$\langle w_0, e_1-e_0 \rangle^\bot \subseteq \langle f_1, \dotsc f_t, e_0, e_1 \rangle = \mathbb{Z}^{t+2}.$$
Since $\tilde{v}\cdot e_0=-\tilde{w}\cdot e_0 = 1$, it is clear from the definition that $c$ is a marked crossing for this embedding.
\end{proof}
Now we give an explicit example to show the tangle replacement of Proposition~\ref{prop:CMidentifiestangle} in action.
\begin{example}Let $L$ be the $107/5$-changemaker lattice, given by
$$L=\langle 4f_3+2f_2+f_1+e_0, e_1-e_0, e_2+e_3-e_1 \rangle^\bot \subset \mathbb{Z}^{7}.$$
This is isomorphic to the Goeritz form of the alternating knot $11a_{15}$. Figure~\ref{fig:11a15example} shows an alternating diagram $D$ of $11a_{15}$ with the induced labeling on the white regions. There is a fractional tangle in $D$ and, as expected, it is of slope $\frac{2}{3}$. Replacing the fractional tangle by a single crossing, $c$, we obtain an alternating diagram $D'$. As shown in Figure~\ref{fig:11a15example}, there is an embedding of the Goeritz form $\Lambda_{D'}$ into $\mathbb{Z}^5$ which shows that $\Lambda_{D'}$ is isomorphic to the $43/2$-changemaker lattice
$$L'=\langle 4f_3+2f_2+f_1+e_0, e_1-e_0\rangle^\bot \subset \mathbb{Z}^{5}.$$
This isomorphism makes $c$ into a marked crossing.
\end{example}
\begin{figure}[h]
  \centering
  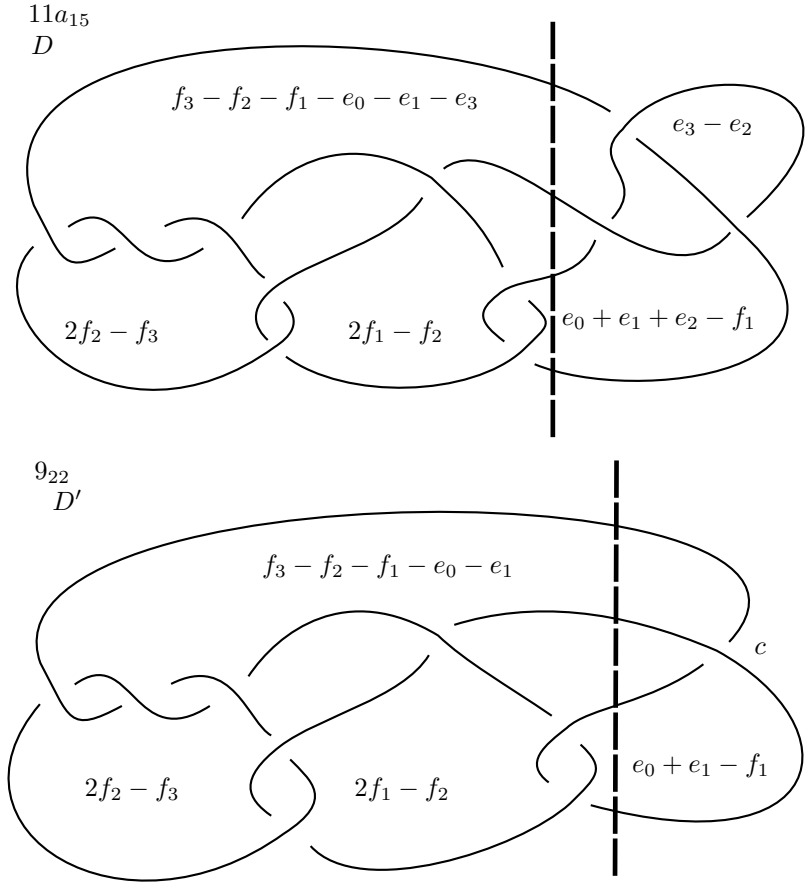
 \caption{A diagram of $11a_{15}$ with the structure of a $107/5$-changemaker lattice. The fractional tangle for this embedding is to the right of the dotted line. Replacing the fractional tangle with a single crossing, $c$, we obtain a diagram of the knot $9_{22}$. The embedding which makes $\Lambda_{D'}$ into a $43/2$-changemaker lattice is also illustrated. It can be checked that $c$ is an unknotting crossing.}
 \label{fig:11a15example}
\end{figure}

\subsection{The main results}
We can now prove our main results.
\begin{proof}[Proof of Theorem~\ref{thm:rationalsurgery}]
First we show that $(i)\Rightarrow (ii)$. Suppose $\Sigma(L)=S_{-p/q}^3(\kappa)$ for some $\kappa \subset S^3$. Let $D$ be a reduced alternating diagram for $L$. Since $D$ is alternating, $\Sigma(L)$ bounds a positive-definite, sharp, simply-connected 4-manifold with intersection lattice isomorphic to $\Lambda_D$ \cite{ozsvath2005heegaard}. Thus, Theorem~\ref{thm:Gibbonsrefined} implies that $\Lambda_D$ is isomorphic to a $p/q$-changemaker lattice.
Applications of Proposition~\ref{prop:CMidentifiestangle}, Lemma~\ref{lem:fractionalslope} and Theorem~\ref{thm:markedmeansunknotting} prove $(ii) \Rightarrow (iii)$.
The implication $(iii) \Rightarrow (i)$ is the Montesinos trick as given by Proposition~\ref{prop:3implies1}.
\end{proof}
\begin{proof}[Proof of Theorem~\ref{thm:informaltheorem}]
By considering $L$ or its reflection $\overline{L}$ as necessary, the theorem follows immediately from the equivalence $(i)\Leftrightarrow(iii)$ in Theorem~\ref{thm:rationalsurgery}.
\end{proof}
\begin{proof}[Proof of Proposition~\ref{prop:smallsurgery}]
If $S^3_{-p/q}(\kappa)$ is the branched double cover of an alternating link, then $\kappa$ is an $L$-space knot as the branched double cover of an alternating link is an $L$-space \cite{ozsvath2005heegaard}. Therefore we must have have the bound \cite{ozsvath2011rationalsurgery}
$$2g(\kappa)-1\leq p/q.$$
If $p/q<1$, then $g(\kappa)=0$. This implies that $\kappa$ is the unknot. The proposition follows since surgery on the unknot yields lens spaces, which are the branched double covers of 2-bridge links.
\end{proof}

\bibliographystyle{plain}
\bibliography{noninteger}

\end{document}